\documentclass[reqno,11pt]{amsart}
\usepackage{amssymb}
\usepackage{verbatim}
\usepackage{color}
\usepackage[pdftex]{graphicx}
\usepackage[pagebackref,pdftex,hyperindex]{hyperref}
\numberwithin{equation}{section}       
\numberwithin{figure}{section}       
\theoremstyle{plain}
\newtheorem{Thm}{Theorem}[section]
\newtheorem{Prop}[Thm]{Proposition}
\newtheorem{Lemma}[Thm]{Lemma}

\newtheorem{Prop-def}[Thm]{Proposition-Definition}

\newtheorem*{ThmD}{Theorem D} 
\newtheorem*{ThmDp}{Theorem D'}

\newtheorem*{ConjA}{Conjecture A} 
\newtheorem*{ConjAp}{Conjecture A'} 
\newtheorem*{ConjB}{Conjecture B} 
\newtheorem*{ConjBp}{Conjecture B'} 
\newtheorem*{ConjC}{Conjecture C} 
\newtheorem*{ConjCp}{Conjecture C'} 
\newtheorem*{ConjCb}{Conjecture C''} 
\newtheorem*{conjectureC''}{Conjecture C''}
\newtheorem*{ConjE}{Conjecture E} 
\newtheorem*{ConjEp}{Conjecture E'}

\theoremstyle{definition}

\newtheorem{Def}[Thm]{Definition}
\newtheorem{Remark}[Thm]{Remark}

\newcommand{\C}{{\mathbf{C}}}
\newcommand{\D}{{\mathbf{D}}}

\renewcommand{\P}{{\mathbf{P}}}
\newcommand{\Q}{{\mathbf{Q}}}
\newcommand{\R}{{\mathbf{R}}}
\newcommand{\Z}{{\mathbf{Z}}}

\newcommand{\bbar}{{\bar{\beta}}}

\newcommand{\cH}{{\mathcal{H}}}

\newcommand{\cJ}{{\mathcal{J}}}

\newcommand{\cO}{{\mathcal{O}}}

\newcommand{\cV}{{\mathcal{V}}}

\newcommand{\fa}{{\mathfrak{a}}}
\newcommand{\fb}{{\mathfrak{b}}}

\newcommand{\fm}{{\mathfrak{m}}}

\newcommand{\fq}{{\mathfrak{q}}}

\newcommand{\tU}{{\tilde{U}}}

\renewcommand{\a}{\alpha}
\renewcommand{\b}{\beta}
\newcommand{\g}{\gamma}
\renewcommand{\d}{\delta}
\newcommand{\la}{\lambda}

\newcommand{\e}{\varepsilon}
\newcommand{\f}{\varphi}
\newcommand{\p}{\psi}

\newcommand{\Lloc}{L^1_{\mathrm{loc}}}
\newcommand{\Ltwo}{L^2}

\newcommand{\eg}{e.g.\ }

\newcommand{\loccit}{\textit{loc.\ cit.\ }}

\newcommand{\Arn}{\operatorname{Arn}}

\newcommand{\lct}{\operatorname{lct}}

\newcommand{\ord}{\operatorname{ord}}

\newcommand{\Spec}{\operatorname{Spec}}

\newcommand{\Val}{\operatorname {Val}}
\newcommand{\Vol}{\operatorname{Vol}}

\newcommand{\llbracket}{[\negthinspace[}
\newcommand{\rrbracket}{]\negthinspace]}

\newcommand{\Dan}{D^{\mathrm{an}}}
\newcommand{\Xan}{X^{\mathrm{an}}}
\newcommand{\Zan}{Z^{\mathrm{an}}}
\newcommand{\pian}{\pi^{\mathrm{an}}}

\title[Openness Conjecture]{An algebraic approach to the Openness Conjecture of Demailly and Koll\'ar}
\date{\today}
\author{Mattias Jonsson and Mircea Musta\c{t}\u{a}}
\address{Dept of Mathematics, University of Michigan, Ann Arbor \\ MI 48109-1043\\ USA}
\email{mattiasj@umich.edu, mmustata@umich.edu}


\thanks{2010\,\emph{Mathematics Subject Classification}.
 Primary 32U05; Secondary 32U25, 14F18, 12J20, 14B05.
\newline The first author was partially supported by
NSF grant DMS-1001740 and the second author 
was partially supported by
 NSF grant DMS-1068190.}

\begin{document}
\begin{abstract}
  We reduce the Openness Conjecture of Demailly and Koll\'ar 
  on the singularities of plurisubharmonic functions to a purely 
  algebraic statement.
\end{abstract}


\maketitle
\setcounter{tocdepth}{1}
\tableofcontents

%
%
%
%
%
%
\section{Introduction}
In this paper we study singularities of plurisubharmonic (psh) functions.
These are important in complex analytic geometry,
see, for instance,~\cite{Lelong,Skoda,Siu74,Dem87,Kis87,Kis94,DK}.
Specifically, we study the \emph{openness conjecture}
of Demailly and Koll\'ar, and reduce this conjecture to a purely
algebraic statement.

\smallskip
Let $\f$ be the germ of a psh function at a point $x$ on a complex
manifold. It is easy to see that the set of real numbers $c>0$ such that 
$\exp(-c\f)$ is locally integrable at $x$ is an interval.
It is nonempty by a result of Skoda~\cite{Skoda}. 
The openness conjecture,
see~\cite[Remark~5.3]{DK}, asserts that this interval is open.
If we define the 
\emph{complex singularity exponent} of $\f$ at $x$ by 
\begin{equation*}
  c_x(\f)=\sup\{c>0\mid\exp(-2c\f)\ \text{is locally integrable at $x$}\},
\end{equation*}
then the conjecture can be stated as follows:
\begin{ConjA}
  If $c_x(\f)<\infty$, then the function $\exp(-2c_x(\f)\f)$ 
  is not locally integrable at $x$.
\end{ConjA}
In fact, Demailly and Koll\'ar made the following slightly more 
precise conjecture, which easily implies Conjecture~A,
see~\cite[Remark~4.4]{DK}.
\begin{ConjB} 
  If $c_x(\f)<\infty$, then for every open neighborhood $U$ of $x$ on
  which $\f$ is defined, we have the estimate
  \begin{equation*}
    \Vol\{y\in U\mid c_x(\f)\f(y)<\log r\}\gtrsim r^2
  \end{equation*}
  as $r\to 0$.
\end{ConjB}
Here and throughout the paper 
we write $f(r)\gtrsim r^2$ as $r\to 0$ if there exists $c>0$
such that $f(r)\ge cr^2$ for all sufficiently small $r$.

Demailly and Koll\'ar also proved that Conjecture~A implies a stronger 
openness statement, namely, that the local integrability of
$\exp(-2\f)$ is an open condition with respect to the $\Lloc$ topology,
see~\cite[Conjecture~5.4]{DK}.

Conjectures~A and~B are easily verified in dimension one. A proof
in the two-dimensional case was given in~\cite{valmul}.
In higher dimensions, they are open. 
In this paper we reduce Conjecture~B to a purely algebraic conjecture:
\begin{ConjC}
  Let $k$ be an algebraically closed field of characteristic zero.
  Let $\fa_\bullet$ be a graded sequence of ideals in
  the polynomial ring $R=k[x_1,\dots,x_m]$ such that $\fa_1$
  is $\fm$-primary, where $\fm$ is a maximal ideal in $R$.
  Then there exists
  a quasimonomial valuation $v$ on $R$ which computes $\lct(\fa_\bullet)$.
\end{ConjC}
Let us briefly explain the meaning of the terms here; see~\S\ref{S116}
and~\cite{graded} for more details.
The \emph{log canonical threshold} $\lct(\fa)$ of an ideal $\fa\subseteq R$ 
is the algebro-geometric analogue of the complex singularity exponent.
A \emph{graded sequence} $\fa_\bullet=(\fa_j)_{j=1}^\infty$ of ideals in $R$ is a sequence
such that $\fa_i\cdot\fa_j\subseteq\fa_{i+j}$
for all $i,j\ge 1$ (we always assume that some $\fa_i$ is nonzero).
We then define
\begin{equation*}
  \lct(\fa_\bullet)
  :=\sup_j j\lct(\fa_j)=\lim_{j\to\infty} j\lct(\fa_j),
\end{equation*}
where the limit is over those $j$ for which $\fa_j$ is nonzero.
Similarly, if
$v\colon R\to\R_{\ge0}$ is a valuation, 
then we set 
\begin{equation*}
  v(\fa_\bullet):=\sup_j j^{-1}v(\fa_j)=
  \lim_{j\to\infty} j^{-1}v(\fa_j),
\end{equation*}
where again the limit is over the $j$ for which $\fa_j$ is nonzero.
One can show that 
\begin{equation}\tag{$\star$}
  \lct(\fa_\bullet)=\inf_v\frac{A(v)}{v(\fa_\bullet)},
\end{equation}
where the infimum is over quasimonomial valuations $v$, that is,
valuations that are monomial in suitable coordinates on a suitable
blowup of $\Spec R$. Here $A(v)$ is the \emph{log discrepancy} of $v$.
Finally, we say that a quasimonomial valuation $v$ \emph{computes}
$\lct(\fa_\bullet)$ if the infimum in~($\star$) is achieved by $v$.

Conjecture~C holds in dimension two, see~\cite[\S9]{graded}. 
In higher dimensions, it is open.
Our main result is
\begin{ThmD}
  If Conjecture~C holds for any dimension $m\le n$ and any
  algebraically closed field $k$ of characteristic zero, then 
  Conjecture~B holds in dimension $n$.
\end{ThmD}

In fact, we prove a slightly more general result, with the log
canonical threshold replaced by more general jumping numbers 
in the sense of~\cite{ELSV}, see Theorem~D' in~\S\ref{S115}.
This result has the following consequence. 
Let $\f$ be a psh function on a complex manifold $U$.
Recall that the \emph{multiplier ideal} $\cJ(\f)$ of $\f$ is the analytic
ideal sheaf on $U$ whose stalk at a point $x$ is given by 
the set of holomorphic germs $f\in\cO_x$ such that 
$|f|^2e^{-2\f}$ is locally integrable at $x$. This is a coherent 
ideal sheaf. Now define $\cJ^+(\f)$ as the (increasing, locally stationary) 
limit of $\cJ((1+\e)\f)$ as $\e\searrow0$.
We then show that a suitable generalization of 
Conjecture~C implies that $\cJ^+(\f)=\cJ(\f)$, 
see Remark~\ref{R102}.

We note that one can formulate a version of Conjecture~C in a more general setting,
dealing with arbitrary graded sequences on regular 
excellent connected schemes over $\Q$.
It was shown in \cite{graded} that this more general conjecture follows from the special
case in Conjecture~C above. 

One can also formulate a similar conjecture for subadditive sequences. Recall that
a sequence $\fb_\bullet=(\fb_j)_{j=1}^\infty$ of nonzero ideals in $R$ is \emph{subadditive} if 
$\fb_i\cdot\fb_j\supseteq\fb_{i+j}$ for all $i,j\ge1$. 
As in the case of graded sequences, we can define $\lct(\fb_{\bullet})$ and $v(\fb_{\bullet})$
when
$v\colon R\to\R_{\ge0}$ is a valuation, and we can consider whether $v$ computes
$\lct(\fb_{\bullet})$. 
We say that $\fb_\bullet$ has \emph{controlled growth} if 
\begin{equation*}
 \frac1jv(\fb_j)\le v(\fb_\bullet)\le\frac1j(v(\fb_j)+A(v))
\end{equation*}
for all $j\ge 1$ and all quasimonomial valuations $v$ on $R$ (with the left inequality being obvious). 
Subadditive systems usually arise as multiplier ideals and then 
are of controlled growth, see Proposition~\ref{P102}
and also~\cite[Proposition~2.13]{graded}. 
We show that Conjecture~C implies (in fact, it is equivalent to) the following statement.

\begin{ConjE}
  Let $\fb_\bullet$ be a subadditive sequence of ideals in
  an excellent regular domain $R$ of equicharacteristic zero.
  If $\fb_\bullet$ is of controlled growth
  and there is a maximal ideal $\fm$ in $R$ and a positive integer $p$ such that
  $\fm^{pj}\subseteq\fb_j$ for all $j$, then there exists
  a quasimonomial valuation $v$ on $R$ such that 
$v$ computes $\lct(\fb_\bullet)$.
\end{ConjE}

It is this form of the conjecture that we will use in the proof of Theorem~D.
Let us now indicate the strategy of this proof.
Suppose $\f$ is a psh germ at a point $x$ on a complex manifold
with $\la:=c_x(\f)<\infty$.
To $\f$ we associate a sequence $\fb_\bullet=(\fb_j)_{j\ge1}$ of ideals 
by letting $\fb_j$ be the analytic multiplier ideal of the 
psh function $j\f$. It follows from~\cite{DEL} that $\fb_\bullet$ is
subadditive. Further, using techniques due to 
Demailly~\cite{Dem92,Dem93}, one can show
that $\fb_\bullet$ has controlled growth and that 
the singularities of $\fb_j$ closely approximate those of $\f$.

The complex singularity exponent $c_y(\f)$ is a lower semicontinuous 
function of the point $y$, so we can define $V$ as the ``log canonical
locus of $\la\f$'', that is, the germ at $x$ of the analytic set defined
by $c_y(\f)\le\la$. 
Assume that $V$ is smooth at $x$ (this is in fact no restriction)
and let $\cO_{x,V}$ be the localization of the ring of holomorphic
germs at $x$, at the prime ideal defined by $V$. 
Since the latter ring is excellent~\cite[Theorem~102]{Matsumura},
so is $\cO_{x,V}$.

We now have the subadditive sequence $\fb_\bullet\cdot\cO_{x,V}$ 
of ideals in the excellent, regular local ring $\cO_{x,V}$.
By applying Conjecture~E, we conclude that there exists a quasimonomial
valuation of $\cO_{x,V}$ computing $\lct(\fb_\bullet\cdot\cO_{x,V})$.
This valuation is monomial in suitable algebraic coordinates
on a regular scheme $X$ admitting a projective birational map to
$\Spec\cO_{x,V}$. We can ``analytify'' the latter map and interpret
the quasimonomial valuation as an analytic invariant, the 
\emph{Kiselman number} of $\f$. Using basic properties of 
psh functions and Kiselman numbers we then obtain the 
desired volume estimates in Conjecture~B.

As already mentioned, Conjecture~C holds in dimension two, so 
we obtain a new proof of the openness conjecture in dimension two.
In fact, this proof is quite similar to the one in~\cite{valmul}. 
The strategy in~\loccit is to consider a subspace $\cV$ of 
semivaluations $v\colon \cO_x\to[0,+\infty]$ satisfying $v(\fm_x)=1$,
where $\fm_x$ is the maximal ideal.
One can equip $\cV$ with a natural topology in which it is compact Hausdorff;
it also has the structure of a tree and is studied in detail 
in~\cite{valtree} (see also~\cite{dynberko}). To a psh germ at $x$
one can associate a lower semicontinuous function on $\cV$ whose minimum is equal
to $c_x(\f)$. It turns out that the minimum must occur for a semivaluation that is
either quasimonomial or associated to the germ of an analytic curve at $x$.
In both cases one can deduce the volume estimate in Conjecture~B
using a simplified version of the arguments in~\S\ref{S117}.

In higher dimensions, the analogue of the space $\cV$ was studied 
in~\cite{hiro}, where it was shown that $c_x(\f)$ can be computed
using quasimonomial valuations. However, when $\f$ does not have an
isolated singularity at the origin, it seems difficult to define a suitable 
lower semicontinuous functional directly on $\cV$, having
minimum equal to $c_x(\f)$.
The idea is to instead work at a generic point of the log canonical locus
of $\la\f$. This does not quite make sense in the analytic category, and for
this reason we pass to algebraic arguments using the subadditive 
sequence $\fb_\bullet$. In the algebraic category, localization arguments 
work quite well and were extensively used in~\cite{graded}.

The idea of studing psh functions using valuations was systematically
developed in~\cite{pshsing,valmul,hiro} but appears already in
the work of Lelong~\cite{Lelong} and Kiselman~\cite{Kis87,Kis94}.
For some recent work on the singularities of psh functions, see 
also~\cite{Rash06,Ber06,Lag10,Gue10}.

\smallskip
The paper is organized as follows. 
In~\S\ref{S102} we review facts about 
sequences of ideals and log canonical thresholds
in an algebraic setting and adapt some of the 
statements to the setting of complex analytic manifolds.
We also prove the equivalence of Conjectures~C and E above.
In~\S\ref{S104} we discuss plurisubharmonic functions,
Kiselman numbers, multiplier ideal sheaves and 
the Demailly approximation procedure.
Finally, the main results are proved in~\S\ref{S103}.

\smallskip
\textbf{Acknowledgment}.
We thank A.~Rashkovskii for spotting a mistake in the proof of
Theorem~D' in an earlier version of the paper. We also thank the referee for 
a careful reading and several useful remarks.
%
%
%
%
%
%
\section{Background}\label{S102}
%
%
%
%
\subsection{Algebraic setting}\label{S116}
We start by recalling some basic algebraic facts.
For more details we refer to~\cite{graded} even though
much of what follows is standard material.
Let $R$ be an excellent, regular domain of equicharacteristic zero. 
In the cases we will consider, $R$ will 
be the localization at a prime ideal of the ring of germs of
holomorphic functions at a point in a complex manifold.
%
%
\subsubsection{Quasimonomial valuations}\label{S110}
By a \emph{valuation} on $R$ we mean a rank 1 valuation
$v\colon R\setminus\{0\}\to\R_{\ge0}$.
A valuation is \emph{divisorial} if there exists
a projective birational morphism $\pi\colon X\to\Spec R$, with $X$ regular,
a prime divisor $E$ on $X$ and a positive number $\a>0$
such that $v=\a\ord_E$, where $\ord_E$ denotes the order of vanishing 
along $E$.

More generally,
consider a projective birational morphism $\pi\colon X\to\Spec R$, with $X$ regular,
a reduced simple normal crossing divisor $E=\sum_{i\in I} D_i$
on $X$, a subset $J\subseteq I$ such that $\bigcap_{i\in J}D_i\ne\emptyset$,
an irreducible component $Z$ of $\bigcap_{i\in J}D_i$ 
and nonnegative numbers $\a_i\ge 0$, $i\in J$, not all zero.
Then there is a unique valuation $v$ on $R$ such that
the following holds: if $(u_j)_{j\in J}$ are local coordinates
at the generic point $\xi$ of $Z$ such that $D_j=\{u_j=0\}$ and we write 
$f\in R\subseteq\cO_{X,\xi}\subseteq\widehat{\cO}_{X,\xi}$ as
$f=\sum_\b c_\b u^\b$, with $c_\b\in\widehat{\cO}_{X,\xi}$ and, for each $\b$,
either $c_\b=0$ or $c_\b(\xi)\ne0$, then
\begin{equation*}
  v(f)=\min\{\sum_i\a_i\b_i\mid c_\beta\ne 0\}.
\end{equation*}
We call such a valuation 
\emph{quasimonomial}\footnote{As opposed to the convention in~\cite{graded}, 
  we do not consider the trivial valuation, which is identically zero on 
  $R\setminus\{0\}$, to be quasimonomial.}
and we say that 
the morphism $\pi\colon X\to\Spec R$ is \emph{adapted} to $v$.
In general, $\pi$ is not unique. 
On the other hand, we can always choose it so
that $\a_i>0$ for all $i$ and that the $\a_i$ are rationally independent,
see~\cite[Lemma~3.6]{graded}.
Finally note that quasimonomial valuations are also known as 
\emph{Abhyankar valuations}, see~\cite{ELS} and~\cite[\S3.2]{graded}.

%
%
\subsubsection{Log discrepancy}
Using the notation above we define the \emph{log discrepancy} $A(v)$ 
of a quasimonomial valuation $v$ by
\begin{equation*}
  A(v)=\sum_{j\in J}\a_j(1+\ord_{D_j}(K)),
\end{equation*}
where $K=K_{X/\Spec R}$ is the relative canonical divisor. 
In particular, $A(\ord_{D_j})=1+\ord_{D_j}(K)$.
One can show that the log discrepancy of a quasimonomial 
valuation does not depend on any choices made,
see~\cite[\S5.1]{graded}. Furthermore, one can extend the definition of
log discrepancy to arbitrary valuations on $R$; in this case the log discrepancy 
can be infinite, see \cite[\S5.2]{graded}.
%
%
\subsubsection{Log canonical thresholds and jumping numbers}\label{S112}
If $\fa\subseteq R$ is a proper nonzero ideal, then we define the 
\emph{log canonical threshold} of $\fa$ as 
\begin{equation}\label{e123}
  \lct(\fa)=\inf_v\frac{A(v)}{v(\fa)},
\end{equation}
where the infimum is over all nonzero valuations on $R$ (it is enough
to only consider quasimonomial or even divisorial valuations).
The quantity $\Arn(\fa)=\lct(\fa)^{-1}$ is called
the \emph{Arnold multiplicity} of $\fa$.
More generally, if $\fq\subseteq R$ is a nonzero ideal, then we define
\begin{equation}\label{e124}
  \lct^\fq(\fa)=\inf_v\frac{A(v)+v(\fq)}{v(\fa)}
  \quad\text{and}\quad
  \Arn^\fq(\fa)=\lct^\fq(\fa)^{-1}.
\end{equation}
Then $\lct^\fq(\fa)$ is a \emph{jumping number} of $\fa$ in the sense of~\cite{ELSV},
and all jumping numbers appear in this way. 
Note that $\lct(\fa)=\lct^R(\fa)$. 
The infimum in~\eqref{e123} (resp.~\eqref{e124}) is
attained at some divisorial valuation associated to a 
prime divisor on some log resolution of $\fa$ (resp.\ $\fa\cdot\fq$).
We make the convention that if $\fa=0$ or
$\fa=R$, then $\lct^{\fq}(\fa)=0$ or $\infty$, respectively.
%
%
\subsubsection{Graded sequences}\label{S1102}
We now recall the definitions of the asymptotic invariants for graded sequences of ideals.
For proofs and details we refer to \cite{graded}, see also~\cite{Musgraded}.
A sequence of ideals $\fa_{\bullet}=(\fa_j)_{j\geq 1}$ is a \emph{graded sequence}
if $\fa_i\cdot\fa_j\subseteq\fa_{i+j}$ for all $i,j\geq 1$. For example, if $v$ is a valuation on $R$
and $\alpha$ is a positive real number, then by putting $\fa_j:=\{f\in R\mid v(f)\geq j\alpha\}$,
we obtain a graded sequence in $R$. We refer to 
\cite[\S10.1]{positivity} for other examples of graded sequences of ideals. 
We assume that all graded sequences are nonzero in the sense that some $\fa_j$ is nonzero.

It follows from the definition that if $\fa_{\bullet}$ is a graded
sequence of ideals in $R$ and $v$ is a valuation on $R$, then 
$v(\fa_{i+j})\leq v(\fa_{i})+v(\fa_{j})$ for all $i,j\geq 1$. 
By Fekete's Lemma, this subadditivity property implies that 
\begin{equation*}
  v(\fa_{\bullet}):=\inf_{j\geq 1}\frac{v(\fa_j)}{j}=\lim_{j\to\infty}\frac{v(\fa_j)}{j},
\end{equation*}
where the limit is over those $j$ such that $\fa_j$ is nonzero.
Similarly, if $\fq$ is a nonzero ideal in $R$, we have
\begin{equation*}
\lct^{\fq}(\fa_{\bullet}):=\sup_{j\geq 1}j\cdot \lct^{\fq}(\fa_j)=\lim_{j\to\infty} j\cdot
\lct^{\fq}(\fa_j),
\end{equation*}
where the limit is over those $j$ such that $\fa_j$ is nonzero. We also put
\begin{equation*}
  \Arn^{\fq}(\fa_{\bullet}):=\lct^{\fq}(\fa_{\bullet})^{-1}.
\end{equation*}
The jumping number $\lct^{\fq}(\fa_{\bullet})$ is positive, but may be infinite.
One can show (see \cite[Corolloray~6.9]{graded}) 
that as in the case of one ideal, we have
\begin{equation}\label{inf_lct}
  \lct^{\fq}(\fa_{\bullet})=\inf_v\frac{A(v)+v(\fq)}{v(\fa_{\bullet})},
\end{equation}
where the infimum is over all nonzero valuations of $R$ 
(it is enough, in fact, to only consider 
quasimonomial or even divisorial valuations). 
 
%
%
\subsubsection{Subadditive sequences}\label{S111}
Let us now review the corresponding notions for the case of subadditive sequences,
referring for details to \cite{graded}. 
A sequence $\fb_\bullet=(\fb_j)_{j\ge0}$ of nonzero ideals in $R$ is called 
\emph{subadditive} if
$\fb_{i+j}\subseteq\fb_i\cdot\fb_j$ for all $i,j\ge0$.
This implies that $v(\fb_{i+j})\ge v(\fb_i)+v(\fb_j)$ for all 
valuations $v$ on $R$ and hence that
\begin{equation*}
  v(\fb_\bullet):=\sup_j\frac{v(\fb_j)}{j}=\lim_{j\to\infty}\frac{v(\fb_j)}{j}\in\R_{\ge0}\cup\{+\infty\}.
\end{equation*}
A subadditive sequence $\fb_\bullet$ has \emph{controlled growth} if 
\begin{equation}\label{e125}
  v(\fb_\bullet)\le\frac{v(\fb_j)}{j}+\frac{A(v)}{j}
\end{equation}
for all $j\ge 1$ and all quasimonomial valuation $v$ (in fact, it is enough to only
impose this condition for divisorial valuations).
In particular, for such $\fb_{\bullet}$ we have $v(\fb_{\bullet})<\infty$
for every quasimonomial valuation $v$. 

For every subadditive system $\fb_\bullet$ and every nonzero ideal $\fq\subseteq R$,
we define
\begin{equation*}
  \lct^\fq(\fb_\bullet):=\inf_{j\geq 1}j\cdot \lct^{\fq}(\fb_j)=\lim_{j\to\infty} j\cdot\lct^{\fq}(\fb_j).
\end{equation*}
We also put $\Arn^\fq(\fb_\bullet)=\lct^\fq(\fb_\bullet)^{-1}$.
For every subadditive sequence we have
\begin{equation}\label{e106v2}
  \lct^\fq(\fb_\bullet)=\inf_v\frac{A(v)+v(\fq)}{v(\fb_\bullet)},
\end{equation}
where the infimum is over all valuations $v$ of $R$ with $A(v)<\infty$ 
(see \cite[Corollary~6.8]{graded}).
It is clear from the definition that $\lct^\fq(\fb_\bullet)<\infty$ 
unless $\fb_j= R$ for all $j$.
Moreover,  if $\fb_\bullet$ has controlled growth, then 
$\lct^\fq(\fb_\bullet)>0$. Indeed, one can easily see that
\begin{equation*}
  \frac1j\Arn^\fq(\fb_j)
  \le\Arn^\fq(\fb_\bullet)
  \le\frac1j\Arn^\fq(\fb_j)+\frac1j
\end{equation*}
for all $j\ge 1$, so that $\Arn^\fq(\fb_\bullet)<\infty$.

Subadditive sequences arise algebraically as asymptotic multiplier ideals.
If $\fa_{\bullet}$ is a graded sequence of ideals in $R$ and if 
$\fb_j={\mathcal J}(\fa_{\bullet}^j)$ is the asymptotic multiplier ideal of
$\fa_{\bullet}$ of exponent $j$, then $\fb_{\bullet}$ is a subadditive sequence of
controlled growth (see \cite[Proposition~2.13]{graded}). Furthermore, we have
$\lct^{\fq}(\fa_{\bullet})=\lct^{\fq}(\fb_{\bullet})$ for every nonzero ideal $\fq$, and
$v(\fa_{\bullet})=v(\fb_{\bullet})$ for every valuation $v$ with $A(v)<\infty$
(see \cite[Proposition~2.14,~6.2]{graded}). 
%
%
\subsubsection{Computing jumping numbers of graded sequences}\label{S113} 
If $\fa_{\bullet}$ is a graded sequence of ideals in $R$ and $\fq$ is a nonzero ideal, 
then we say that a nonzero valuation 
$v$ \emph{computes} $\lct^{\fq}(\fa_{\bullet})$ if $v$ achieves the
infimum in (\ref{inf_lct}), that is, $\lct^{\fq}(\fa_{\bullet})=
\frac{A(v)+v(\fq)}{v(\fa_{\bullet})}$. Note that if $\lct^{\fq}(\fa_{\bullet})=\infty$, then
$v(\fa_{\bullet})=0$ for every $v$; hence every $v$ computes $\lct^{\fq}(\fa_{\bullet})$.
In what follows we will focus on the case $\lct^{\fq}(\fa_{\bullet})<\infty$; 
then every valuation $v$ that computes $\lct^{\fq}(\fa_{\bullet})$ 
must satisfy $A(v)<\infty$. 

It was shown in \cite[Theorem~7.3]{graded} that for every graded sequence $\fa_{\bullet}$
and every nonzero ideal $\fq$,
there is a valuation $v$ on $R$ that computes $\lct^{\fq}(\fa_{\bullet})$. One should contrast this with
\begin{ConjCp}
  For every excellent regular domain $R$ of equicharacteristic zero,
  every graded sequence of ideals $\fa_{\bullet}$, and every
  nonzero ideal $\fq$ in $R$, there is a quasimonomial valuation $v$ on $R$ that computes
  $\lct^{\fq}(\fa_{\bullet})$.
\end{ConjCp}

One can also consider the following special case of this conjecture.
\begin{ConjCb}
  If $k$ is an algebraically closed field of characteristic zero and $R=k[x_1,\ldots,x_m]$,
  then for every nonzero ideal $\fq$ and every graded sequence $\fa_{\bullet}$ of ideals
  in $R$ such that $\fa_1\supseteq\fm^p$ for some maximal ideal $\fm$ in $R$ and some
  $p\geq 1$, there is a quasimonomial valuation $v$ on $R$ that computes
  $\lct^{\fq}(\fa_{\bullet})$.
\end{ConjCb}

Note that when the ideal $\fq$ is equal to $R$, Conjecture~C''
specializes to Conjecture~C in the Introduction. 
It is shown\footnote{For the comparison with~\cite{graded}, note that a ring $R$
has equicharacteristic zero iff $\Spec R$ is a scheme over $\Q$.}
in~\cite[Theorem~7.6]{graded} that Conjecture~C' holds for rings of
dimension $\leq n$ if and only if Conjecture C'' holds for rings of dimension 
$\leq n$.
We note that both  conjectures are trivially true in dimension one.  They are
 also true in dimension two.
A proof, modeled on ideas in~\cite{valmul} is given in~\cite[\S9]{graded}.
%
%
\subsubsection{Computing jumping numbers of subadditive sequences}
We now turn to the analogous considerations for subadditive sequences.
If $\fb_{\bullet}$ is such a sequence,
then a valuation $v$ with $A(v)<\infty$ computes $\lct^{\fq}(\fb_{\bullet})$
if $v$ achieves the infimum in (\ref{e106v2}).  
The following conjecture extends Conjecture~E from the Introduction to the case 
of arbitrary jumping numbers.
\begin{ConjEp}
  Let  $\fq$ be a nonzero ideal and $\fb_\bullet$ a subadditive sequence of ideals in
  an excellent regular domain $R$ of equicharacteristic zero.
  If $\fb_\bullet$ is of controlled growth
  and there is a maximal ideal $\fm$ in $R$ and a positive integer $p$ such that
  $\fm^{pj}\subseteq\fb_j$ for all $j$, then there exists
  a quasimonomial valuation $v$ on $R$ such that 
  $v$ computes $\lct^{\fq}(\fb_\bullet)$.
\end{ConjEp}

The key requirement in the above conjecture is that the valuation $v$ be
quasimonomial. The next proposition shows that if we drop this requirement, we can find 
a valuation computing the log canonical threshold. This is the analogue of the corresponding result for graded sequences that we have mentioned above. 

\begin{Prop}\label{existence_subadditive}
  Under the assumptions in Conjecture~E', there is a nonzero valuation $v$ on $R$ with
  $A(v)<\infty$ which computes $\lct^{\fq}(\fb_{\bullet})$. 
\end{Prop}
\begin{Remark}
  A similar result appears in~\cite{Hu12a}, see also~\cite{Hu12b}.
\end{Remark}
\begin{proof}
  The argument follows verbatim the proof of \cite[Thm~7.3]{graded}, 
  which treated the case of graded sequences.\footnote{The proof of
    \cite[Thm~7.3]{graded} involves an extra step, the reduction to the case when all ideals are $\fm$-primary, for some maximal ideal $\fm$ in $R$; in our case,
    we do not have to worry about this step since this is part of the hypothesis.}
  If $\lct^{\fq}(\fb_{\bullet})=\infty$, then the assertion is
  trivial: we may take $v$ to be any quasimonomial valuation such that
  $v(\fm)=0$, since in this case $v(\fb_{\bullet})=0$.
  Hence, from now on, we may assume that $\lct^{\fq}(\fb_{\bullet})<\infty$. 
  
  By the assumption on $\fb_{\bullet}$, if $v(\fm)=0$, then $v(\fb_{\bullet})=0$. 
  Therefore we only need to focus on valuations $v$ with
  $v(\fm)>0$, and, after normalizing,
  we may assume that $v(\fm)=1$. Let us fix $\epsilon$ with
  $0<\epsilon<\Arn^{\fq}(\fb_{\bullet})$ and suppose that 
  $\frac{v(\fb_{\bullet})}{A(v)+v(\fq)}>\epsilon$. 
  For every $j\geq 1$ we have $\fm^{pj}\subseteq \fb_j$,
  hence $v(\fb_j)\leq jp$, and therefore $v(\fb_{\bullet})\leq p$. This implies that
  $A(v)\leq A(v)+v(\fq)\leq M$, where $M=p/\epsilon$. We thus have
  \begin{equation*}
    \lct^{\fq}(\fb_{\bullet})=\inf_{v\in V_M}\frac{A(v)+v(\fq)}{v(\fb_{\bullet})},
  \end{equation*}
  where $V_M$ is the set of all valuations $v$ with $v(\fm)=1$ and $A(v)\leq M$. 
  
  The space of all valuations carries a natural topology, and the subspace $V_M$
  is compact by \cite[Proposition~5.9]{graded}. 
  Moreover, $A$ is a lower semicontinuous function on $V_M$, while the functions 
  $v\mapsto v(\fq)$ and $v\mapsto v(\fb_{\bullet})$ are continuous on $V_M$ 
  by~\cite[Proposition~5.7,~Corollary~6.6]{graded} (for the last assertion we make use 
  of the hypothesis that $\fb_{\bullet}$ has controlled growth). 
  It follows that the function $v\mapsto\frac{A(v)+v(\fq)}{v(\fb_{\bullet})}$ is
  lower semicontinuous on $V_M$; hence it
  achieves its infimum at some point $v\in V_M$.
\end{proof}
%
%
\subsubsection{Equivalence of conjectures}
Our main result is that Conjecture~C' implies the openness conjecture, 
see Theorem~D' in~\S\ref{S115}. 
As a first step, we show that Conjectures~C',~C'' and E' are equivalent.
\begin{Prop}\label{P101}
  If one of Conjectures~C',C'', and E' holds for all rings of
  dimension $\leq n$,  then the other two conjectures hold for such rings.
\end{Prop}
\begin{proof}
  As we have already mentioned, \cite[Theorem~7.6]{graded} gives the equivalence of 
  Conjectures C' and C''. On the other hand, it is easy to see that if Conjecture~E'
  holds in dimension $\leq n$, then so does Conjecture~C''. Indeed, let $\fq$ and
  $\fa_{\bullet}$
  be as in Conjecture C'', and let $\fb_j=\cJ(\fa_{\bullet}^j)$. In this case $\fb_{\bullet}$
  is a subadditive sequence of controlled growth, and for every $j\geq 1$ we have
  \begin{equation*}
    \fm^{pj}\subseteq \fa_1^j\subseteq\fa_j\subseteq\fb_j.
  \end{equation*}
  Furthermore, if $v$ is a quasimonomial valuation of $R$ which computes
  $\lct^{\fq}(\fb_{\bullet})$, then, since $v(\fa_{\bullet})=v(\fb_{\bullet})$ and
  $\lct^{\fq}(\fa_{\bullet})=\lct^{\fq}(\fb_{\bullet})$, it follows that $v$ computes
  $\lct^{\fq}(\fa_{\bullet})$. Therefore Conjecture~C'' holds in dimension $\leq n$. 
  
  We now assume that Conjecture~C' holds in dimension $\leq n$, and consider 
  a nonzero $\fq$ and a subadditive sequence $\fb_{\bullet}$ as in Conjecture~E',
  with $\dim(R)\leq n$. We may assume that $\lct^{\fq}(\fb_{\bullet})<\infty$, since otherwise the
  assertion to be proved is trivial (note also that $\lct^{\fq}(\fb_{\bullet})>0$ since
  $\fb_{\bullet}$ has controlled growth).
  It follows from Proposition~\ref{existence_subadditive} that 
  there is a nonzero valuation $w$ of $R$ with $A(w)<\infty$ which computes
  $\lct^{\fq}(\fb_{\bullet})$. In particular, $w(\fb_{\bullet})$ is finite and positive.
  If we put $\fa_j=\{f\in R\mid w(f)\geq j\}$, then $\fa_{\bullet}$ 
  is a graded sequence of ideals, and by
  Conjecture~C' there is a quasimonomial valuation $v$ on $R$ which computes
  $\lct^{\fq}(\fa_{\bullet})$. It is enough to show that in this case $v$ also computes
  $\lct^{\fq}(\fb_{\bullet})$. 

  It follows easily from the definition of $\fa_{\bullet}$ that 
  $w(\fa_{\bullet})=1$ and if $v(\fa_{\bullet})=\gamma$, then
  \begin{equation*}
    \gamma=\inf\{v(f)/w(f)\mid f\in R, w(f)>0\}
  \end{equation*}
  (see for example \cite[Lemma~2.4]{graded}).
  We first deduce that $\lct^{\fq}(\fa_{\bullet})\leq A(w)+w(\fq)<\infty$, hence
  $\gamma>0$. Furthermore, we have $\gamma^{-1}v\geq w$.
  Since $v$ computes $\lct^{\fq}(\fa_{\bullet})$, we have
  \begin{equation}\label{eq2_P101}
    \gamma^{-1}(A(v)+v(\fq))=\frac{A(v)+v(\fq)}{v(\fa_{\bullet})}
    \leq \frac{A(w)+w(\fq)}{w(\fa_{\bullet})}=A(w)+w(\fq).
  \end{equation}
  On the other hand, using the fact that $\gamma^{-1}v\geq w$, we obtain 
  $\gamma^{-1}v(\fb_{\bullet})\geq w(\fb_{\bullet})$, and therefore (\ref{eq2_P101})
  gives
  \begin{equation}\label{eq3_P101}
    \frac{A(v)+v(\fq)}{v(\fb_{\bullet})}\leq \frac{A(w)+w(\fq)}{w(\fb_{\bullet})}.
  \end{equation}
  By assumption, $w$ computes $\lct^{\fq}(\fb_{\bullet})$, hence
  we have equality in (\ref{eq3_P101}), and $v$ also computes 
  $\lct^{\fq}(\fb_{\bullet})$.
\end{proof}
\begin{Remark}\label{special_equivalence}
  By running the argument in the proof of Proposition~\ref{P101} with $\fq=R$, we see that 
  Conjectures~C and E in the Introduction are equivalent, in the sense that one holds for
  rings of dimension $\leq n$ if and only if the other one does.
\end{Remark}
%
%
\subsubsection{A converse to the conjectures}
As a partial converse to the (equivalent) conjectures~C',~C'' and~E' 
above we show
that any quasimonomial valuation computes \emph{some} jumping number.
This result will not be used in the sequel. In its formulation and
proof we freely use terminology from~\cite{graded}.
\begin{Prop}
  Let $X$ be an excellent, regular, connected, separated scheme over $\Q$
  and let $v$ be a quasimonomial valuation on $X$. 
  Then there exists a nonzero ideal $\fq$ on $X$ and a graded sequence 
  $\fa_\bullet$ on $X$ such that $v$ computes $\lct^\fq(\fa_\bullet)$.
\end{Prop}
\begin{proof}
  By~\cite[Thm~7.8]{graded} it suffices to find a nonzero ideal 
  $\fq$ such that
  the following statement holds: for every valuation $w\in\Val_X$ such
  that $w\ge v$ (in the sense that $w(\fa)\ge v(\fa)$ for all ideals
  $\fa$ on $X$) we have $A(w)+w(\fq)\ge A(v)+v(\fq)$.
  Here $A=A_X$ is the log discrepancy with respect to $X$.

  After replacing $X$ by an open neighborhood of the center 
  $c_X(v)$ of $v$ on $X$ we may assume that $X=\Spec R$ is 
  affine. 
  Since $v$ is quasimonomial, there exists a proper birational 
  morphism $\pi\colon Y\to X$, with $Y$ regular,
  and algebraic local coordinates
  $y_1,\dots,y_n$ at $c_Y(v)$ with respect to which $v$ is monomial.
  Let $E_i=\{y_i=0\}$, $1\le i\le n$, be the associated prime divisors
  on $Y$ and pick $N$ large enough so that $N\ge A(\ord_{E_i})$ for all
  $i$.  Also write $y_i=a_i/b_i$ with $a_i,b_i\in R$ nonzero.

  We claim that the principal ideal $\fq=(b_1\cdot\ldots\cdot b_n)^N$ 
  does the job.
  Indeed, suppose $w\in\Val_X$ satisfies $w\ge v$.
  In particular, we then have
  \begin{equation}\label{e129}
    w(a_i)\ge v(a_i)
    \quad\text{and}\quad
    w(b_i)\ge v(b_i)
    \quad\text{for all $i$}.
  \end{equation}
  Since $v$ is monomial in coordinates $y_1,\dots,y_n$ we have 
  \begin{equation}\label{e130}
    A(v)=\sum_{i=1}^nv(y_i)A(\ord_{E_i}).
  \end{equation}
  By the definition of $A(w)$ we also have 
  \begin{equation}\label{e131}
    A(w)\ge\sum_{i=1}^nw(y_i)A(\ord_{E_i}).
  \end{equation}
 Equations~\eqref{e130} and~\eqref{e131} and the definition of $\fq$
 now imply
 \begin{multline*}
   A(w)-A(v)+w(\fq)-v(\fq)
   \ge\sum_{i=1}^nA(\ord_{E_i})(w(y_i)-v(y_i))
   +\sum_{i=1}^nN(w(b_i)-v(b_i))\\
   =\sum_{i=1}^nA(\ord_{E_i})(w(a_i)-v(a_i))
   +\sum_{i=1}^n(N-A(\ord_{E_i}))(w(b_i)-v(b_i))
   \ge0,
 \end{multline*}
 where the last inequality follows from~\eqref{e129} and the choice of
 $N$.
 This completes the proof.
\end{proof}
\begin{Remark}
  It follows from~\cite[Thm~7.8]{graded} that 
  with the choice of $\fq$ above, $v$ also computes
  $\lct^\fq(\fb_\bullet)$ for some subadditive sequence 
  $\fb_\bullet$ as well as $\lct^\fq(\fa'_\bullet)$, where
  $\fa'_\bullet$ is the graded sequence defined by 
  $\fa'_j=\{v\ge j\}$ for $j\ge 1$.
\end{Remark} 
%
%
%
\subsection{Analytic setting}\label{S105}
Let $U$ be a complex manifold. 
When talking about open sets in $U$ we always refer
to the classical topology unless mentioned otherwise.
By an \emph{ideal} on $U$ we will always 
mean a coherent analytic ideal sheaf on $U$. 
For a point $x\in U$, $\cO_x$ denotes the
ring of germs of holomorphic functions at $x$.
Note that $\cO_{x}$ is isomorphic to the ring of convergent power series
in $n$ variables over ${\mathbf C}$, where $n=\dim(U)$; hence $\cO_x$ is an 
excellent regular local ring, see~\cite[Thm~102]{Matsumura}.

We denote by $\fm_x$ the maximal ideal in $\cO_x$.
By a \emph{valuation} at $x$ we mean a valuation on $\cO_x$
in the sense of~\S\ref{S110}.

A \emph{subadditive sequence} of ideals on $U$ is a sequence
$\fb_\bullet=(\fb_j)_{j=1}^\infty$ of everywhere nonzero ideals on $U$ such that 
$\fb_i\cdot\fb_j\supseteq\fb_{i+j}$. 
If $x\in U$, then we write 
$\fb_\bullet\cdot\cO_x$ for the corresponding subadditive sequence
inside $\cO_x$. 
We say that $\fb_\bullet$ has \emph{controlled growth} if
$\fb_\bullet\cdot\cO_x$ has controlled growth for all $x\in U$.

If $\fq$ is an everywhere nonzero ideal on $U$, $\fb_\bullet$ is a subadditive sequence of
ideals on $U$ and $x\in U$, then we define 
$\lct_x^\fq(\fb_\bullet):=\lct^{\fq\cdot\cO_x}(\fb_\bullet\cdot\cO_x)$ and 
$\Arn_x^\fq(\fb_\bullet):=\lct^\fq_x(\fb_\bullet)^{-1}$. Thus we have
\begin{equation*}
  \Arn_x^\fq(\fb_\bullet)=\sup_v\frac{v(\fb_\bullet)}{A(v)+v(\fq)},
\end{equation*}
where the supremum is over all (quasimonomial) valuations at $x$
(note that we simply write $v(\fb_{\bullet})$ and $v(\fq)$ for
$v(\fb_{\bullet}\cdot\cO_x)$ and $v(\fq\cdot\cO_x)$, respectively).

More generally, we shall consider the following situation.
Let $V$ be a germ of a complex submanifold at a point $x$ in 
a complex manifold.
Let $\cO_{x,V}$ be the localization of $\cO_x$ along the ideal $I_V$.
This is an excellent regular local ring with maximal ideal $\fm_{x,V}$.
Consider a subadditive system of ideals
$\fb_\bullet$ defined near $x$ and a nonzero ideal $\fq\subseteq\cO_x$.
We set
\begin{equation*}
  \lct^\fq_{x,V}(\fb_\bullet):=\lct^{\fq\cdot\cO_{x,V}}(\fb_\bullet\cdot\cO_{x,V})
  \quad\text{and}\quad
  \Arn^\fq_{x,V}(\fb_\bullet):=\lct^\fq_{x,V}(\fb_\bullet)^{-1}
\end{equation*}
Then
\begin{equation*}
  \Arn_{x,V}^\fq(\fb_\bullet)=\sup_v\frac{v(\fb_\bullet)}{A(v)+v(\fq)},
\end{equation*}
where the supremum is over all (quasimonomial) valuations 
of $\cO_{x,V}$.

Note that if $V=\{x\}$, then we recover the previous situation.
%
%
\subsubsection{Analytification of birational morphisms}\label{S108}
Let $x$, $V$ be as above.
Consider a projective birational morphism
$\pi\colon X\to\Spec\cO_{x,V}$ of schemes over $\C$, 
with $X$ regular.
We can then analytify $\pi$ as follows.\footnote{The analytification procedure here is ad hoc and not functorial, but nevertheless related to the construction of a complex manifold associated to a smooth complex projective variety.}
Since $\pi$ is projective, there exists a closed embedding
\begin{equation}\label{e105}
  X\hookrightarrow \Spec\cO_{x,V}\times_{\Spec\C}\P^N_\C
\end{equation}
such that $\pi$ is the restriction of the projection of
the right hand side onto $\Spec\cO_{x,V}$. Thus $X$ is 
cut out by finitely many homogeneous equations 
with coefficients in $\cO_{x,V}$. These coefficients can 
be written as $f_i/g$, where $f_i\in\cO_x$ 
and $g\in\cO_x\setminus I_V\cdot\cO_x$.
Let $U$ be a neighborhood of $x$ on which 
$g$ and the $f_i$ are defined. Set $W:=\{g=0\}\subseteq U$. 
This is a (possibly empty)
analytic subset of $U$ that does not contain $V$.
After shrinking $U$ we may assume that 
$W$ is either empty or contains $x$.

We now define a complex manifold 
\begin{equation*}
  \Xan\hookrightarrow (U\setminus W)\times\P^N(\C)
\end{equation*}
as the analytic subset cut out by the same equations as 
in~\eqref{e105}. Then $\Xan$ is a complex manifold and 
the induced projection $\pian\colon\Xan\to U\setminus W$ is 
a proper modification. 
We shall have more to say about this construction later.

Given a point $y\in U\setminus W$ we define in the same way a 
projective birational morphism $\pi_y\colon X_y\to\Spec\cO_y$.
Namely, $X_y\subseteq\Spec\cO_y\times_{\Spec\C}\P^N_\C$ 
is defined by the same homogeneous polynomials as above.
After shrinking $U$ and increasing $W$ 
(but keeping $x\in U$ and $W\not\supseteq V$) 
we may further obtain that if $E\subseteq\pi_y^{-1}(V)$ is any 
prime divisor, then the image $\pi_y(E)$ contains $V$.
Further, suppose $\fb$ is an ideal on $U$ and that 
$\pi\colon X\to\Spec\cO_{x,V}$ as above is a log resolution of $\fb\cdot\cO_{x,V}$. 
Then, we may assume that the birational morphism
$\pi_y\colon X_y\to\Spec\cO_y$ is a log resolution of $\fb\cdot\cO_y$
for all $y\in U\setminus W$. 
Further, there is a bijection between the set of prime divisors $E$ of $X_y$
for which $\ord_E(\fb\cdot\cO_y)>0$ and the set of prime divisors 
$E$ of $X$ such that 
$\ord_E(\fb\cdot\cO_{x,V})>0$.
%
%
\subsubsection{Log canonical locus}\label{S109}
The key to the proof of Theorem~D is to localize at the locus where 
the log canonical threshold is as small as possible.
Let $x$ and $V$ be as above and let
$\fb_\bullet$ (resp.\ $\fq$) be a subadditive
system of ideals (resp.\ a nonzero ideal) defined on some neighborhood $U$ of $x$.
Assume that $U$ is small enough that $V$ is a submanifold of $U$.
\begin{Lemma}\label{L101}
  Assume that $\fb_\bullet$ has controlled growth
  and that
  \begin{equation*}
    \lct^\fq_y(\fb_\bullet)\le\la=\lct^q_x(\fb_\bullet)
  \end{equation*}
  for all $y\in V$, where $\la\ge0$. 
  Then $\lct^\fq_{x,V}(\fb_\bullet)=\la$.
\end{Lemma}
\begin{proof}
  Let us first prove that $\lct_{x,V}^\fq(\fb_\bullet)\ge\la$.
  For this, fix $m\ge1$ and pick a log resolution
  $\pi_m\colon X_m\to\Spec\cO_x$
  of the ideal $(\fq\fb_m)\cdot\cO_x$.
  After a base change by $\Spec\cO_{x,V}\to\Spec\cO_x$, $\pi_m$
  induces a log resolution 
  of the ideal $(\fq\fb_m)\cdot\cO_{x,V}$.
  We have 
  \begin{equation*}
    \Arn_x^\fq(\fb_m)
    =\max_E\frac{\ord_E(\fb_m)}{A(\ord_E)+\ord_E(\fq)},
  \end{equation*}  
  where the maximum is over the set of 
  prime divisors $E\subseteq X_m$ for which $\ord_E(\fb_m)>0$.
  On the other hand 
  $\Arn_{x,V}^\fq(\fb_m)$ is given by the same expression, 
  but where the maximum is only over the subset of 
  prime divisors $E$ for which 
  $\pi_m(E)$ contains $V$. It is then clear that 
  $\Arn_{x,V}^\fq(\fb_m)\le\Arn_x^\fq(\fb_m)$.
  Dividing by $m$ and letting $m\to\infty$ yields 
  $\Arn_{x,V}^\fq(\fb_\bullet)\le\Arn_x^\fq(\fb_\bullet)$
  and hence
  $\lct_{x,V}^\fq(\fb_\bullet)\ge\lct_x^\fq(\fb_\bullet)=\la$.

  Now we prove the reverse inequality.
  Pick any $m\ge 1$. Consider a log resolution 
  $\pi_m\colon X_m\to\Spec\cO_{x,V}$ 
  of the ideal $\fq\fb_m\cdot\cO_{x,V}$. 
  As in~\S\ref{S108} this gives rise to an open neighborhood $U_m$
  of $x$, an analytic subset $W_m\subseteq U_m$ not containing $V\cap U_m$,
  and for each $y\in V\cap (U_m\setminus W_m)$, a
  log resolution $\pi_{m,y}\colon X_{m,y}\to\Spec\cO_y$ of
  $\fq\fb_m\cdot\cO_y$. 
  Further, there is a bijection between the set of prime divisors 
  $E\subseteq X_{m,y}$ such that $\ord_E(\fb_m\cdot\cO_y)>0$ and the 
  set of prime divisors $E\subset X_m$ for which $\ord_E(\fb_m\cdot\cO_{x,V})>0$.
  This implies that 
  $\Arn_y^\fq(\fb_m)=\Arn_{x,V}^\fq(\fb_m)$
  for any $y\in (U_m\setminus W_m)\cap V$
  since both quantities are calculated using the same divisors.
  Thus we have 
  \begin{equation*}
    \Arn_y^\fq(\fb_m)=\Arn_{x,V}^\fq(\fb_m)\le m\Arn_{x,V}^\fq(\fb_\bullet),
  \end{equation*}
  the inequality being definitional.

  Now, for any $y\in(U_m\setminus W_m)\cap V$ and any $j\ge1$ 
  there exists a quasimonomial (or even divisorial) valuation
  $v_j$ at $y$ such that 
  \begin{equation*}
    \frac{v_j(\fb_\bullet)}{A(v_j)+v_j(\fq)}\ge\Arn_y^\fq(\fb_\bullet)-\frac1{j}.
  \end{equation*}
  This gives
  \begin{multline*}
    \la^{-1}
    \le\Arn_y^\fq(\fb_\bullet)
    \le\frac{v_j(\fb_\bullet)}{A(v_j)+v_j(\fq)}+\frac1{j}
    \le\frac{v_j(\fb_m)+A(v_j)}{m(A(v_j)+v_j(\fq))}+\frac1{j}\\
    \le\frac{v_j(\fb_m)}{m(A(v_j)+v_j(\fq))}+\frac1{m}+\frac1j
    \le\frac{\Arn_y^\fq(\fb_m)}{m}+\frac1{m}+\frac1j
    \le\Arn_{x,V}^\fq(\fb_\bullet)+\frac1{m}+\frac1j,
  \end{multline*}
  where the second inequality follows from the assumption that 
  $\fb_\bullet\cdot\cO_y$  has controlled growth.
  Letting first $j\to\infty$ and then $m\to\infty$ we obtain
  $\Arn_{x,V}^\fq(\fb_\bullet)\ge\la^{-1}$,
  and hence $\lct_{x,V}^\fq(\fb_\bullet)\le\la$.
  This completes the proof.
\end{proof}
%
%
%
%
\section{Plurisubharmonic functions}\label{S104}
Let $U$ be a complex manifold. 
A function $\f\colon U\to\R\cup\{-\infty\}$ is 
\emph{plurisubharmonic} (psh) if $\f\not\equiv-\infty$ on any
connected component of $U$, 
$\f$ is upper semicontinuous, and 
$\tau^*\f$ is subharmonic or $\equiv-\infty$ for
every holomorphic map $\tau\colon \D\to U$, where $\D\subseteq\C$ is the unit
disc.
A \emph{germ} of a psh function at a point is defined in the obvious way. 

A basic example of a psh function is $\f=\log\max_i|f_i|$, where 
$f_1,\dots,f_m$ are holomorphic functions on $U$. If $x\in U$
and $\fq\subseteq\cO_x$ is an ideal, then we can define a psh
germ $\log|\fq|$ at $x$ by setting 
\begin{equation}\label{e101}
  \log|\fq|:=\log\max_i|f_i|
\end{equation}
where $f_i\in\cO_x$ are generators of $\fq$. 
The choice of generators only affects $\log|\fq|$ up to 
a bounded additive term. 
If $\fa$ and $\fb$ are two ideals in $\cO_x$, then 
\begin{equation*}
  \log|\fa\cdot\fb|=\log|\fa|+\log|\fb|+O(1).
\end{equation*}
For further facts about psh functions, see~\eg~\cite{Hormander,DemBook}.
%
%
\subsection{Jumping numbers and singularity exponents}\label{S115}
Let $\f$ be a psh germ at a point $x$ on a complex manifold 
and let $\fq\subseteq\cO_x$ be a nonzero ideal. Define
\begin{equation*}
  c_x^\fq(\f)=\sup\{c>0\mid|\fq|^2\exp(-2c\f)\ \text{is locally integrable at $x$}\}.
\end{equation*}
This definition does not depend on the choice of generators used to define $|\fq|$.
We also write
$c_x(\f):=c_x^{\cO_x}(\f)$. In~\cite{DK}, $c_x(\f)$ is called the 
\emph{complex singularity exponent} of $\f$ at $x$,
whereas $c_x^\fq(\f)$ is a \emph{jumping number} 
in the sense of~\cite{ELSV}.
If $\fa\subseteq\cO_x$ is a nonzero ideal, then we have 
\begin{equation*}
  c_x^\fq(\log|\fa|)=\lct^\fq_x(\fa),
\end{equation*}
where the right hand side is defined as in~\S\ref{S105},
see~\cite[Proposition~1.7]{DK}.\footnote{In~\loccit only the case $\fq=\cO_x$ is treated but the same proof works in the general case.}
We have the following generalizations of Conjectures~A and~B.
\begin{ConjAp}
  If $c_x^{\fq}(\f)<\infty$, then the  function $|\fq|^2\exp(-2c_x^\fq(\f)\f)$ is not locally integrable at $x$.
\end{ConjAp}
This conjecture, which is also due to Demailly and Koll\'ar,
can be paraphrased as a semicontinuity statement for 
multiplier ideals, see Remark~\ref{R102}.
\begin{ConjBp}
  If $c_x^{\fq}(\f)<\infty$, then for any open neighborhood $U$ of $x$ on which $\f$ and $\fq$ 
  are defined, we have
  \begin{equation}\label{e117}
    \Vol\{y\in U\mid c_x(\f)\f(y)-\log|\fq|<\log r\}\gtrsim r^2
  \end{equation} 
  as $r\to 0$.
\end{ConjBp}
It is clear that Conjecture~B' implies Conjecture~A' and that neither conjecture
depends on the choice of generators of $\fq$.
The following result is a variation of Theorem~D from the introduction.
We shall prove both these theorems in~\S\ref{S103}. 

\begin{ThmDp}
  If Conjecture~C'' in~\S\ref{S113} holds for all $m\le n$ 
  and all algebraically closed fields
  $k$ of characteristic zero,
  then Conjecture~B' holds on complex manifolds of dimension $n$.
\end{ThmDp}

%
%
\subsection{Kiselman numbers}\label{S107}
We now recall an analytic version of monomial valuations
due to Kiselman~\cite{Kis87,Kis94}. It is a
special case of the generalized Lelong numbers introduced by
Demailly~\cite{Dem87}.
As our setting differs slightly from the above references, 
we give some details for the convenience of the reader.

Let $\Omega$ be a complex manifold of dimension $n$, 
$Z\subseteq\Omega$ a connected submanifold of codimension $m\ge 1$
and $D_1,\dots,D_m$ distinct, smooth, connected hypersurfaces 
in $\Omega$ such that $Z=\bigcap_i D_i$ and such that 
the $D_i$ meet transversely along $Z$.
Also suppose we are given positive real numbers $\a_i>0$, $1\le i\le m$.
In this situation we will associate to any psh function $\f$ on $U$
its \emph{Kiselman number} $\tau_{Z,D,\a}(\f)\ge0$. 

In preparation for the definition, pick a point $z\in Z$ 
and local analytic coordinates $(u_1,\dots,u_n)$ at $z$ 
such that $D_i=\{u_i=0\}$ for $1\le i\le m$ locally at $z$.
For $t\in\R^n_{\le0}$ with $t_i\ll0$ 
let $D_u(t)\subseteq\Omega$ be the polydisc with radius $e^t$ and
$S_u(t)$ its distinguished boundary, that is
\begin{equation}
  D_u(t):=\bigcap_{i=1}^n\{|u_i|\le e^{t_i}\}
  \quad\text{and}\quad
  S_u(t):=\bigcap_{i=1}^n\{|u_i|=e^{t_i}\}.
\end{equation}
We also write $D_u(s)=D_u(s,s,\dots,s)$ for $s\in\R_{\le0}$.

Let $\f$ be a psh germ at $z$ and pick $\e>0$ small enough that 
$\f$ is defined in an open neighborhood of the polydisc $D_u(\log\e)$.
For $t\in\R^n_{\le\log\e}$ set
\begin{equation*}
  H(t):=\sup_{D_u(t)}\f=\sup_{S_u(t)}\f.
\end{equation*}
Clearly $H$ is increasing in each argument and it is 
finite-valued since $\f$ is upper semicontinuous.
Less obvious is the fact that $H$ is convex; 
for this see~\cite[p.12]{Kis94}.
Note that $H$ is continuous on the closed set $\R^n_{\le\log\e}$ since it is
defined and convex on an open neighborhood of this set.

We now define a new function
\begin{equation*}
  h=h_{\f,Z,D,z,u,\e}\colon \R^n_{\ge0}\to\R
\end{equation*}
by setting
\begin{equation*}
  h(\a):=\lim_{s\to-\infty}\frac{H(\log\e+s\a)}{s}.
\end{equation*}
The limit is well defined by the convexity of $H$.
\begin{Lemma}\label{L102}
  The function $h=h_{\f,Z,D,z,u,\e}$ has the following properties:
  \begin{itemize}
  \item[(i)]
    $h$ is nonnegative, continuous, concave, 1-homogeneous and 
    increasing in each argument;
  \item[(ii)]
    $h$ does not depend on the choice of $\e$ as long as $\f$ is defined 
    in an open neighborhood of $D(\log\e)$.
  \end{itemize}
  If, further, $\a_i=0$ for $i>m$, then 
  \begin{itemize}
  \item[(iii)]
    $h$ does not depend on the choice of local coordinates $(u_1,\dots,u_n)$
    at $z$, as long as $D_i=\{u_i=0\}$ for $i\le m$;
  \item[(iv)]
    $h$ does not depend on the choice of point $z\in Z$ 
    as long as $\f$ is defined in a neighborhood of $z$.
  \end{itemize}
\end{Lemma}
\begin{proof}
  To alleviate notation, we shall only write out the relevant part of the
  subscripts of $h=h_{\f,Z,D,z,u,\e}$.
  
  The fact that $h$ is nonnegative, continuous, concave and increasing follows
  from $H$ being continuous, convex and increasing.
  That $h$ is 1-homogeneous is clear. This proves~(i).

  As for~(ii), suppose $0<\e'<\e$. It is clear that $h_{\e'}\ge h_\e$ since $H$ is 
  increasing. To prove the reverse
  inequality, first suppose that $\a_i>0$ for $1\le i\le n$ and set 
  $\d=\min_i\a_i$. Then 
  \begin{equation*}
    D_u(\log\e+s\a)\subseteq D_u(\log\e'+(s+\frac{1}{\d}\log\frac{\e}{\e'})\a),
  \end{equation*}
  so that $H(\log\e+s\a)\le H(\log\e'+(s+\frac1{\d}\log\frac{\e}{\e'})\a)$ 
  for any $s\le 0$.
  This implies $h_{\e}(\a)\ge h_\e(\a)$. By continuity of $h_\e$ and $h_{\e'}$
  we get $h_\e\ge h_{\e'}$ and hence $h_\e=h_{\e'}$ on $\R^n_{\ge0}$.
  
  Now we turn to~(iii) and~(iv) so suppose $\a_i=0$ for $i>m$.

  Let $(u'_1,\dots,u'_n)$ be another set of local analytic
  coordinates at $z$ such that $D_i=\{u'_i=0\}$
  for $1\le i\le m$. 
  We can write $u_i=u'_ig_i$ for $1\le i\le m$, where $g_i\in\cO_z$ and 
  $g_i(z)\ne0$. 
  It is easy to see that if $\e>0$ is small enough, 
  then there exists $\e'>0$
  such that 
  \begin{equation*}
    D_{u'}(\log\e'+s\a)\subseteq D_u(\log\e+s\a)
  \end{equation*}
  for all $s\le0$.
  This gives $h_{u'}(\a)=h_{u',\e'}(\a)\ge h_{u,\e}(\a)=h_u(\a)$ and the 
  reverse inequality follows by symmetry. Thus~(iii) holds.

  Finally we prove~(iv). Thus pick a point $z\in Z$,
  and a set of local coordinates $u$ at $z$.
  Pick $0<\e\ll1$ and
  $z'\in D_u(\log\e)\cap Z$. 
  Then $u':=u-u(z')$ defines local coordinates at $z'$ 
  and 
  for any $\a$ as above and any $s\le 0$ we have 
  \begin{equation*}
    D_{u'}(\log\e+s\a)\subseteq D_{u}(\log2\e+s\a)
    \quad\text{and}\quad
    D_u(\log\e+s\a)\subseteq D_{u'}(\log2\e+s\a).
  \end{equation*}
  This implies that 
  $h_{z'}=h_{z',u',\e}\ge h_{z,u,2\e}=h_z$ and, similarly, $h_z\ge h_{z'}$.
  Thus $z\mapsto h_z$ is locally constant on $Z$, which completes
  the proof since $Z$ is connected.
\end{proof}
Now assume $\a_i=0$ for $i>m$ and $\a_i>0$ for $1\le i\le m$.
The number 
\begin{equation*}
  \tau_{Z,D,\a}(\f):=h_{\f,Z,D}(\a)
\end{equation*}
is called the 
\emph{Kiselman number}\footnote{In~\cite{Kis87,Kis94}, the Kiselman number 
  is called \emph{a refined Lelong number} whereas 
  Demailly~\cite{Dem87} calls it a \emph{directional Lelong number}.}
of $\f$ along $Z$ with weight $\a_i$ along $D_i$.
As explained in Lemma~\ref{L102}, it does not depend
on the choice of coordinates $u_i$ defining the hypersurfaces 
$D_i$. However, given such coordinates, it follows
from the convexity of $H$ that we have the estimate
\begin{equation}\label{e102}
  \f\le \tau_{Z,D,\a}(\f)\max_{i\le m}\frac1{\a_i}\log|u_i|+O(1),
\end{equation}
near $z$. From this inequality we easily deduce
\begin{Lemma}\label{L103}
  Suppose $\f$, $\p$ are psh functions defined near some 
  Zariski general point $z\in Z$. Write $\tau=\tau_{Z,D,\a}$.
  Then:
  \begin{itemize}
  \item[(i)]
    if $\f\le\p+O(1)$ near $z$, then $\tau(\f)\ge\tau(\p)$;
  \item[(ii)]
    $\tau(\max\{\f,\p\})=\min\{\tau(\f),\tau(\f)\}$.
  \end{itemize}
\end{Lemma}
\begin{proof}
  The inequality in~(i) follows immediately from the definition.
  As for~(ii), note that~(i) implies
  $\tau(\max\{\f,\p\})\le\min\{\tau(\f),\tau(\f)\}$. 
  The reverse inequality follows from~\eqref{e102}.
\end{proof}
\begin{Remark}
  It is also true that $\tau(\f+\p)=\tau(\f)+\tau(\p)$, but we do not 
  need this fact.
\end{Remark}
\begin{Remark}
  Using the same construction, we can define $\tau_{Z,D,\a}$ when 
  $Z$ and the $D_i$ are germs of complex submanifolds at a point 
  in a complex manifold.
\end{Remark}
\begin{Remark}
  When $\a_i=1$ for $1\le i\le m$, the choice of hypersurfaces $D_i$
  play no role; in this case the Kiselman number is equal to the
  \emph{Lelong number} along $Z$~\cite{Lelong}. 
\end{Remark}
%
%
\subsection{Kiselman numbers and quasimonomial valuations}\label{S114}
Let $U$ be a complex manifold, $x\in U$ a point and 
$V$ the germ at $x$ of a complex submanifold of $U$.
We allow for the case $V=x$ but assume that $V$ has codimension $\ge 1$.
As in~\S\ref{S108} let $\cO_{x,V}$ be the localization of $\cO_x$ at the ideal $I_V\cdot\cO_x$
and let $\fm_{x,V}$ be the maximal ideal of $\cO_{x,V}$.
Let $v$ be a quasimonomial valuation of $\cO_{x,V}$.
We want to associate to $v$ a Kiselman number on a suitable modification.

Consider a projective birational morphism $\pi\colon X\to\Spec\cO_{x,V}$ that is
adapted to $v$ in the sense of~\S\ref{S110}.
Thus there exist prime divisors $D_1,\dots,D_m$ on $X$ such that $\sum D_i$ 
has simple normal crossing singularities, and an irreducible component $Z$ 
of $\bigcap_{i=1}^mD_i$ such that $v$ is monomial with weight $\a_i>0$ on $D_i$
for $1\le i\le m$. 
The assumption that $v\ge 0$ on $\cO_{x,V}$
implies that $\pi(Z)\subset V$.
Let $\xi$ be the generic point of $Z$ and pick functions 
$u_i\in\cO_{X,\xi}$, $1\le i\le m$, such that $D_i=(u_i=0)$.
Thus the functions $u_i$ are regular on a Zariski open subset of $Z$.

Using the construction and conventions of~\S\ref{S108}, 
after shrinking $U$ a little, the projective birational morphism
$\pi\colon X\to\Spec\cO_{x,V}$ 
gives rise to a complex manifold
$\Xan$ and a proper modification 
$\pian\colon\Xan\to U\setminus W$, where 
$W\subseteq U$ is a complex subvariety not containing $V$. 

Further, after again shrinking $U$ and increasing $W$ if necessary,
there exists an open subset $\Omega$ of $\Xan$ 
on which the functions $u_i$, $1\le i\le m$, are holomorphic
and such that the following properties hold: 
the sets $\Dan_i:=(u_i=0)$ are complex submanifolds of $\Omega$ 
of codimension one, meeting transversely along the connected
submanifold $\Zan:=\bigcap_{1\le i\le m}\Dan_i$. 
Further $\pian(\Zan)\supseteq V\setminus W$.

Let $\tau=\tau_{\Zan,\Dan,\a}$ denote the Kiselman number with respect to the data above,
see~\S\ref{S107}.
\begin{Def}
  If $\f$ is the germ of a psh function at $x$, 
  then we define
  \begin{equation}\label{e110}
    v(\f):=\tau(\f\circ\pian).
  \end{equation}
\end{Def}
Note that this definition \textit{a priori} depends on a lot
of choices made above. However, we have:
\begin{Prop}\label{P103}
  The definition of $v(\f)$ does not depend on any choices made
  as long as the birational morphism $\pi\colon X\to\Spec\cO_{x,V}$
  is adapted to $v$.
\end{Prop}
We shall prove this result 
in~\S\ref{S106} using multiplier ideals, see Remark~\ref{R101}.
For now, we only treat the following special case.
\begin{Lemma}\label{L104}
  If $\fb\subseteq\cO_x$ is a nonzero ideal, then 
  \begin{equation}\label{e109}
    v(\fb\cdot\cO_{x,V})=\tau((\log|\fb|)\circ\pian).
  \end{equation}
\end{Lemma}
\begin{proof}
  Note that both sides of~\eqref{e109} depend continuously on the weight
  $\a\in\R^m_{>0}$. Hence we may assume that the $\a_i$ are
  rationally independent.  In view of Lemma~\ref{L103} we may also 
  assume that $\fb$ is generated by a single element $f\in\cO_x$.
  We must prove that $\tau(\log|f|\circ\pian)=v(f)$.

  Consider a Zariski general closed point $z\in Z$ and 
  pick functions $u_{m+1},\dots,u_n\in\cO_{X,z}$ such 
  that $u:=(u_1,\dots,u_n)$ define local algebraic coordinates on $X$
  at $z$.
  Write $u':=(u_1,\dots,u_m)$ and $u'':=(u_{m+1},\dots,u_n)$. 
  Consider the expansion of $f\circ\pi$ as a formal power series
  in $\widehat{\cO_{X,\xi}}\simeq\C\llbracket u\rrbracket=\C\llbracket u''\rrbracket\llbracket u'\rrbracket$:
  \begin{equation}\label{e114}
    f\circ\pi
    =\sum_{\b\in\Z_{\ge0}^m,\g\in\Z_{\ge0}^{n-m}}a_{\b\g}(u')^\b(u'')^\g
    =\sum_{\b\in\Z_{\ge0}^m}a_\b(u'')(u')^\b,
  \end{equation}
  where $a_{\b,\g}\in\C$ and 
  \begin{equation}\label{e301}
    a_\b(u'')
    =\sum_{\g\in\Z_{\ge0}^{n-m}}a_{\b\g}(u'')^\g
    \in\C\llbracket u''\rrbracket\subseteq\C\llbracket u\rrbracket.
  \end{equation}
  Since the $\a_i$ are rationally independent, there exists a unique
  $\bbar$ minimizing $\b\cdot\a:=\b_1\a_1+\dots+\b_m\a_m$ over all $\b$
  for which $a_\b\not\equiv0$. By definition, we then have 
  \begin{equation*}
    v(f)=\bbar\cdot\a.
  \end{equation*}

  Since the point $z\in Z$ was generically chosen, it corresponds
  to a point, also denoted $z$, on the complex manifold $\Zan$.
  We may assume that such that $f$ is holomorphic near
  $\pian(z)\in U\setminus W$. 
  Pick $0<\e\ll1$ such that $f\circ\pian$ is holomorphic on the open 
  polydisk $|u_i|<\e$, $1\le i\le n$.
  The first series in~\eqref{e114} is then the Taylor series
  of the holomorphic function $f\circ\pian$ at $z$ in the analytic
  coordinates $u$ and this series converges locally uniformly in the polydisk $\|u\|<\e$.
  Further, for every $\b$, the series in~\eqref{e301} converges locally
  uniformly for $\|u''\|<\e$ to a holomorphic function $a_\b(u'')$ and 
  the second series in~\eqref{e114} converges locally uniformly for 
  $\|u\|<\e$. 

  By assumption, the holomorphic function $a_{\bbar}$ is not constantly equal to zero.
  After moving $z$ (but keeping $z\in\Zan$) a little and translating the coordinates 
  $u_{m+1},\dots,u_n$ accordingly, we may assume that 
  $a_{\bbar}(0)\ne0$. Let us use the notation of~\S\ref{S107}. 
  For $0<\e\ll1$ we have 
  \begin{equation*}
    \log|f|\circ\pian
    =\log|f\circ\pian|
    \sim\log|a_{\bbar}(0)(u')^{\bbar}|\sim s \bbar\cdot\a
  \end{equation*}
  on the set $S_u(\log\e+s\a)$, as $s\to-\infty$. 
  This implies that 
  \begin{equation*}
    \tau(\log|f|\circ\pian)=\bbar\cdot\a=v(f)
  \end{equation*}
   as was to be shown.
\end{proof}
%
%
\subsection{Multiplier ideal sheaves and Demailly approximation}\label{S106}
To a psh function $\f$ on a complex manifold $U$ is associated
a \emph{multiplier ideal sheaf} $\cJ(\f)$.
This is an ideal sheaf on $U$
whose stalk at a point $x$ is the set of holomorphic germs
$f\in\cO_x$ such that $|f|^2e^{-2\f}$ is locally integrable at $x$.
The coherence of $\cJ(\f)$ is a nontrivial result due to Nadel~\cite{Nadel89,Nadel90},
which can be proved using H\"ormander's $\Ltwo$-estimates,
see~\cite[Thm~4.1]{DK}. 

Recall from~\S\ref{S115} the definition of the 
jumping number $c_x^\fq(\f)$ of $\f$ at $x$
relative to an ideal $\fq$ on $U$. 
Given $\mu>0$ consider the 
colon ideal $a_{\mu}=(\cJ(\mu\f):\fq)$ on $U$.
This is an ideal sheaf on $U$ whose 
stalk at a point $x\in U$ is given by
\begin{equation}\label{e126}
  \fa_\mu\cdot\cO_x
  :=\{h\in\cO_x\mid
  |h|^2|\fq|^2e^{-2\mu\f}\
  \text{is locally integrable at $x$}\}.
\end{equation}
Since $\cJ(\mu\f)$ and $\fq$ are coherent, so is $a_\mu$.
\begin{Lemma}\label{L106}
  We have $c_x^\fq(\f)<\mu$ iff $\fa_\mu\cdot\cO_x\ne\cO_x$.
  As a consequence, the function $x\mapsto c_x^\fq$ is 
  lower semicontinuous in the 
  analytic Zariski topology on $U$.
\end{Lemma}
\begin{proof}
  The first statement is clear.
  Hence, for $\mu>0$, the set 
  \begin{equation}\label{e127}
    V_\mu^-:=\{x\in U\mid c_x^\fq(\f)<\mu\}
  \end{equation}
  is equal to the support of the coherent sheaf $\cO_U/\fa_\mu$ and
  in particular an analytic subset of $U$. 
  It follows that for $\la>0$, the set 
  \begin{equation}\label{e128}
    V_\la:=
    \{x\in U\mid c_x^\fq(\f)\le\la\}
    =\bigcap_{\mu>\la}V_\mu^-
  \end{equation}
  is also an analytic subset of $U$.
  This concludes the proof.
\end{proof}
\begin{Remark}\label{R102}
  Conjecture~A' in~\S\ref{S115} is equivalent to a 
  semicontinuity statement about multiplier ideals.
  Indeed, define $\cJ^+(\f)$ as the increasing (locally stationary) limit of 
  $\cJ((1+\e)\f)$ as $\e\searrow0$.
  Then Conjecture~A' precisely says that $\cJ^+(\f)=\cJ(\f)$.
\end{Remark}
If $f_1,\dots,f_m$ are holomorphic functions on $U$, generating an ideal sheaf 
$\fq$ and if $\log|\fq|$ is the corresponding psh function on $U$
defined by~\eqref{e101}, then we have 
\begin{equation}\label{e113}
  \cJ(\log|\fq|)=\cJ(\fq),
\end{equation}
where the right-hand side is defined as in~\S\ref{S116},
see~\cite[Proposition~1.7]{DK}. 
\begin{Lemma}\label{L105}
  If $\f\ge p\log|\fq|+O(1)$ for some integer $p\ge 1$, 
  then 
  $\cJ(\f)\supseteq\fq^p$.
\end{Lemma}
\begin{proof}
  In view of the assumption and~\eqref{e113} we have
  \begin{equation*}
    \cJ(\f)
    \supseteq\cJ(p\log|\fq|)
    =\cJ(\fq^p)
    \supseteq\fq^p.
  \end{equation*}
  Here the last inclusion 
  holds since $\cJ(\fa)\supseteq\fa$ for any ideal $\fa$.
\end{proof}
Now fix a psh function $\f$ on $U$. For $j\ge1$ set
\begin{equation*}
  \fb_j:=\cJ(j\f).
\end{equation*}
It follows from~\cite{DEL} that $\fb_\bullet=(\fb_j)_{j=1}^\infty$ 
is a subadditive sequence of ideals on $U$.
The following result (which was known in the case $V=x$, 
see~\cite[Theorem~4.2]{DK} 
and~\cite[Theorem~5.5]{hiro})
allows us to understand the singularities of
$\f$ in terms of those of $\fb_\bullet$. 
\begin{Prop}\label{P102}
  Let $x$ be any point in $U$ and let $V$ be the germ at $x$ of a 
  proper complex submanifold. Define $\cO_{x,V}$ as in~\S\ref{S105}.
  Then the following properties hold:
 \begin{itemize}
    \item[(i)]
      for every nonzero ideal $\fq\subseteq\cO_x$ we have
      $c_x^\fq(\f)=\lct_x^\fq(\fb_\bullet)$;
    \item[(ii)]
      the subadditive sequence $\fb_\bullet\cdot\cO_{x,V}$ has controlled growth;
    \item[(iii)]
      for every quasimonomial valuation $v$ on $\cO_{x,V}$ we have 
      $v(\f)=v(\fb_\bullet\cdot\cO_{x,V})$.
    \end{itemize}      
\end{Prop}
\begin{Remark}\label{R101}
  In~(iii) we compute $v(\f)$ as a Kiselman number of the pullback of $\f$ 
  under a suitable proper modification, the latter being 
  the analytification of a blowup of $\Spec\cO_{x,V}$, see~\S\ref{S114}. 
  Since the quantity $v(\fb_\bullet\cdot\cO_{x,V})$ does not depend on 
  any choices made, we see that $v(\f)$ is uniquely defined.
  Thus we obtain a proof of Proposition~\ref{P103}.
\end{Remark}

The proof of Proposition~\ref{P102} 
relies on a fundamental approximation procedure due to
Demailly~\cite{Dem92,Dem93}. 
We refer to~\cite[\S4]{DK} for details on what follows.

Let $\f$ be a psh function defined in some pseudoconvex domain 
$B\subseteq U$ containing $x$. For $p\ge1$ consider the Hilbert space
\begin{equation*}
  \cH_p:=\{f\in\cO(B)\mid\int_B|f|^2e^{-2p\f}<\infty\},
\end{equation*}
with the natural inner product.
It is a fact that for every $y\in B$, the elements of $\cH_p$ generate the 
stalk at $y$ of the multiplier ideal sheaf $\fb_p:=\cJ(p\f)$.
Define
\begin{equation*}
  \f_p:=\frac1p\sup\{\log|f|\mid \int_B|f|^2e^{-2p\f}\le 1\}.
\end{equation*}
Then $\f_p$ is psh on $B$. 
It follows from the Ohsawa-Takegoshi Theorem that
\begin{equation}\label{e107}
  \f\le\f_p+\frac{C}{p}
\end{equation}
on $B$, for some constant $C$ not depending on $\f$ or $p$.
For any $y\in B$ and any nonzero ideal $\fq\subseteq\cO_y$ we also have
\begin{multline}\label{e108}
  (p\lct^\fq_y(\fb_p))^{-1}
  =c_y^\fq(\f_p)^{-1}
  \le c_y^\fq(\f)^{-1}\\
  \le c_y^\fq(\f_p)^{-1}+\frac1p
  =(p\lct^\fq_y(\fb_p))^{-1}+\frac1p.
\end{multline}
Here the two equalities follow from~\eqref{e113} whereas
the first inequality results from~\eqref{e107}.
The second inequality is proved in~\cite[Thm~4.2~(3)]{DK}
in the case $\fq=\cO_U$ and the same proof works in the general case.

\begin{proof}[Proof of Proposition~\ref{P102}]
  Clearly~(i) follows from~\eqref{e108} with $y=x$ by letting $p\to\infty$.
  It remains to prove~(ii) and~(iii).

  We use the notation of~\S\ref{S114}.
  Write $\tau=\tau_{Z,D,\a}$. 
  It follows from~\eqref{e107} and from Proposition~\ref{P103} that
  \begin{equation}
    \tau(\f\circ\pian)
    \ge\tau(\f_p\circ\pian)
    =\frac1pv(\fb_p\cdot\cO_{x,V})
  \end{equation}
  for any $p\ge1$.

  We will show that if $f\in\cH_p$, then 
  \begin{equation}\label{e111}
    v(f)+A(v)\ge p\tau(\f\circ\pian).
  \end{equation}    
  Grant~\eqref{e111} for the moment. We then have
  \begin{equation}\label{e112}
    \frac1p v(\fb_p\cdot\cO_{x,V})
    \le\tau(\f\circ\pian)
    \le\frac1pv(\fb_p\cdot\cO_{x,V})+\frac1pA(v).
  \end{equation}
  Letting $p$ tend to infinity we see that 
  $\tau(\f\circ\pian)=v(\fb_\bullet\cdot\cO_{x,V})$, proving~(iii).
  In particular, $v(\f)=\tau(\f\circ\pian)$ is well defined independently
  of any choices made so we have established Proposition~\ref{P103}.
  Since $v$ was an arbitrary quasimonomial valuation on $\cO_{x,V}$ we also
  see that $\fb_\bullet\cdot\cO_{x,V}$ has controlled growth, proving~(ii).

  It only remains to prove~\eqref{e111}. 
  Since both
  sides of~\eqref{e111} depend continuously on the weight $\a$,
  we may assume that $\a_1,\dots,\a_m$ are rationally independent.
  We now argue as in the proof of Lemma~\ref{L104}, recycling the
  notation from that proof. 
  Thus we have the expansion~\eqref{e114} and we have 
  $a_{\bbar}(0)\ne0$ for the unique $\bbar\in\Z^m_{\ge0}$ for which 
  $\bbar\cdot\a=v(f)$.
  
  Now fix $K\gg1$. Define a sequence of disjoint open subsets 
  $(\Omega_k)_{k\ge0}$ of $\Omega$ by 
  \begin{equation}\label{e118}
    \Omega_k
    :=\bigcap_{i=1}^m\{-(K+k)<\frac{\log|u_i|}{\a_i}<1-(K+k)\}
    \cap\bigcap_{i=m+1}^n\{-K<\log|u_i|<1-K\}.
  \end{equation}
  For large $k$ we then have the following estimates on 
  $\Omega_k$:
  \begin{equation*}
    \log|f\circ\pian|\ge-kv(f)+O(1)
    \quad\text{and}\quad
    \f\circ\pian\le-k\tau(\f\circ\pian)+O(1).
  \end{equation*}
  Here the second estimate follows from~\eqref{e102}.
  Let $\eta$ be a nonvanishing holomorphic volume form near $x$ and write 
  $(\pian)^*\eta=J\pian\cdot\eta_u$ near $z$, where $\eta_u:=du_1\wedge\dots\wedge du_n$.
  We then have
  \begin{equation}\label{e122}
    \log|J\pian|=\sum_{i=1}^m(A_i-1)\log|u_i|+O(1)
  \end{equation}
  near $z$, where $A_i\in\Z_{>0}$. Further, we have 
  \begin{equation}
    A(v)=\sum_{i=1}^m\a_iA_i.
  \end{equation}
  As $k\to\infty$, we then have 
  \begin{equation}\label{e115}
    \log|J\pian|\sim-k(A(v)-\sum_{i=1}^m\a_i)+O(1)
  \end{equation}
  on $\Omega_k$.
  Moreover, the volume of $\Omega_k$ can be estimated by 
  \begin{equation}\label{e116}
    \log\int_{\Omega_k}(\sqrt{-1})^n\eta_u\wedge\overline{\eta_u}
    =-2k\sum_{i=1}^m\a_i+O(1)
  \end{equation}
  as $k\to\infty$.  
  
 Note that if $K$ is large enough, then 
 $\pian$ is biholomorphic on $\Omega_k$
 and $\pian(\Omega_k)$ is contained in $B$
 for all $k\ge0$.
 Thus we get
 \begin{multline*}
   +\infty
   >\int_B|f|^2e^{-2p\f}(\sqrt{-1})^n\eta\wedge\overline{\eta}
   \ge\sum_{k=0}^\infty\int_{\pian(\Omega_k)}|f|^2e^{-2p\f}(\sqrt{-1})^n\eta\wedge\overline{\eta}\\
   =\sum_{k=0}^\infty
   \int_{\Omega_k}|f\circ\pian|^2e^{-2p\f\circ\pian}|J\pian|^2(\sqrt{-1})^n\eta_u\wedge\overline{\eta_u}\\
   \gtrsim\sum_{k=0}^\infty
   \exp\left(
     -2kv(f)+2kp\tau(\f\circ\pian)-k\left(A(v)-\sum_{i=1}^m\a_i\right)
   \right)
   \int_{\Omega_k}(\sqrt{-1})^n\eta_u\wedge\overline{\eta_u}\\
   \gtrsim\sum_{k=0}^\infty
   \exp\left(-2k(v(f)-p\tau(\f\circ\pian)+A(v))\right),
 \end{multline*}    
 which yields~\eqref{e111} (with strict inequality).
\end{proof}
%
%
%
%
%
%
\section{Proof of the main results}\label{S103}
We are now ready to prove Theorem~D from the introduction and 
its variant Theorem~D' from~\S\ref{S115}.
Consider a germ of a psh function $\f$ at a point $x$ in a 
complex manifold of dimension $n$ and let $\fq\subseteq\cO_x$ be a
nonzero ideal such that $c_x^{\fq}(\f)<\infty$.
Let $U$ be a small open neighborhood of $x$ such that 
$\f$ and $\fq$ are defined on an open neighborhood of $\overline{U}$.
Also fix a nonvanishing holomorphic volume form $\eta$ in a neighborhood of 
$\overline{U}$ and compute all volumes with respect to the positive 
measure $(\sqrt{-1})^n\eta\wedge\overline\eta$.
%
%
%
%
\subsection{Analytic reduction}
As in~\S\ref{S106} set
\begin{equation*}
  V_\mu:=\{y\in U\mid c^\fq_y(\f)\le\mu\}
\end{equation*}
for $\mu\ge\la:=c_x^\fq(\f)$ and
\begin{equation*}
  V:=V_\la=\{y\in U\mid c^\fq_y(\f)\le\la\}.
\end{equation*}
Note that $x\in V$.
By the lower semicontinuity of $y\mapsto c^\fq_y$ (see Lemma~\ref{L106}),
$V_\mu$ is a proper analytic subset of $U$ for any $\mu\ge\la$
and $V$ is the decreasing intersection of $V_\mu$ for all $\mu>\la$.
Using the fact that $\f$ and $\fq$ are defined in a neighborhood of
$\overline{U}$ we deduce the existence of $\mu>\la$ such that $V=V_\mu$.
\begin{Lemma}\label{L201}
  In order to prove Theorem~D', it suffices to assume that $V$ is
  smooth at $x$ and that  $\f\ge p\log|I_V|+O(1)$ near $x$ for some integer 
  $p\ge0$.
\end{Lemma}
\begin{proof}
  We can replace $x$ by a Zariski general point in $V$.
  Indeed, we have $c_y^\fq(\f)=\la$ for a
  Zariski general point $y\in V$, and if the estimate 
  \begin{equation}
    \Vol\{y'\in U_y\mid \la\f(y')-\log|\fq|<\log r\}\gtrsim r^2
  \end{equation} 
  holds for every neighborhood $U_y$ of any point $y$ in
  a dense subset of $V$, then it also holds for every neighborhood
  of $x$.
  In particular, we may assume that $V$ is smooth at $x$. 

  Pick generators of $I_V\cdot\cO_x$. After shrinking $U$, 
  we may assume these generators are defined on $U$
  and that the associated psh function $\log|I_V|$, defined 
  as in~\eqref{e101}, is negative on $U$.
  For an integer $p>0$ define 
  \begin{equation*}
    \tilde\f:=\max\{\f,p\log|I_V|\}.
  \end{equation*}
  We claim that $c_x^\fq(\tilde{\f})=c_x^\fq(\f)$ for $p\gg0$.
  This will allow us to replace $\f$ by $\tilde\f$ and complete the
  proof.  Indeed, we have $\f\le\tilde\f$ so if the estimate~\eqref{e117}
  holds with $\f$ replaced by $\tilde\f$, then it must also hold for $\f$.

  To prove the claim, pick $\mu>\la$ such that $V_{\mu}=V$.
  Consider the colon ideal $a_{\mu}=(\cJ(\mu\f):\fq)$ on $U$.
  This is a coherent ideal sheaf on $U$ whose 
  stalk at $y\in U$ is given by
  \begin{equation*}
    \fa_\mu\cdot\cO_y
    :=\{h\in\cO_y\mid
    |h|^2|\fq|^2e^{-2\mu\f}\
    \text{is locally integrable at $y$}\}.
  \end{equation*}
  The fact that $V_\mu=V$ implies that
  the zero locus of $\fa_\mu$ is equal to $V$. 
  Hence the Nullstellensatz implies that there exists
  $N\ge1$ such that $I_V^N\subseteq\fa_\mu$.
  Now pick the integer $p>0$ large enough so that 
  \begin{equation*}
    p>\frac{N}{\mu-\la}.
  \end{equation*}

  Pick any $\la'\in(\la,\mu)$ such that $p>N/(\mu-\la')$. 
  For $0<r\ll1$ define
  Borel subsets $U_r$, $\tU_r$ and $U'_r$ of $U$ by 
  \begin{align*}
    U_r&:=\{\la'\f-\log|\fq|<\log r\}\\
    \tU_r&:=\{\la'\tilde{\f}-\log|\fq|<\log r\}\\
    U'_r&:=\{\mu\f-\log|\fq|-N\log|I_V|<\log r\}.
  \end{align*}
  It follows from the choice of $p$ that $U_r\subseteq \tU_r\cup U'_r$.
  The inclusion $I_V^N\subseteq\fa_\mu$ 
  guarantees that, after possibly shrinking $U$, we have
  \begin{equation}
    \int_0^\infty\Vol(U'_r)\frac{dr}{r^3}<\infty.\label{e201}
  \end{equation}
  Indeed, if we set $F:=\exp(N\log|I_V|+\log|\fq|-\mu\f)$,
  then, after shrinking $U$, 
  \begin{equation*}
    \infty
    >\int_U F^2
    =2\int_0^\infty \Vol(U\cap\{F>t\})t\,dt
    =2\int_0^\infty \Vol(U'_r)\frac{dr}{r^3},
  \end{equation*}
  where the last equality follows from setting $t=1/r$.

  On the other hand, the fact that $\la'>\la=c_x^\fq(\f)$ implies
  that   
  \begin{equation*}
    \int_0^\infty\Vol(U_r)\frac{dr}{r^3}=\infty.
  \end{equation*}
  The inclusion $U_r\subseteq \tU_r\cup U'_r$ then gives
  \begin{equation*}
    \int_0^\infty\Vol(\tU_r)\frac{dr}{r^3}=\infty,
  \end{equation*}
  so that $c_x^\fq(\tilde\f)\le\la'$. 
  Letting $\la'\to\la$ we get  $c_x^\fq(\tilde\f)\le\la$. 
  But $\tilde\f\ge\f$, so we must have
  $c_x^\fq(\tilde\f)\ge c_x^\fq(\f)=\la$ and hence 
  $c_x^\fq(\tilde\f)=c_x^\fq(\f)$, establishing the claim
  and completing the proof of the lemma.
\end{proof}
\begin{Remark}
  The proof of Lemma~\ref{L201} can be viewed as
  an analytic analogue of the arguments in~\cite[\S7.4]{graded}.
\end{Remark}
%
%
%
%
\subsection{End of proof}\label{S117}
Let $x$ and $V$ be as above. In particular, $V$ is smooth at $x$.
Let $\cO_{x,V}$ be the localization of $\cO_x$
at the ideal $I_V\cdot\cO_x$. Then $\cO_{x,V}$ is a regular local
ring with maximal ideal $\fm_{x,V}=I_V\cdot\cO_{x,V}$.
Its dimension is equal to the codimension of $V$ and hence
bounded by $n$. It is also an excellent ring.
Indeed, $\cO_x$ is isomorphic to the ring of convergent power series 
in $n$ variables, hence excellent,
see~\cite[Theorem~102]{Matsumura},
and excellence is preserved by localization.

Set $\fb_j=\cJ(j\f)$ for $j\ge 0$. 
Then $\fb_\bullet\cdot\cO_{x,V}$ is a subadditive system of ideals having 
controlled growth, see Proposition~\ref{P102}. 
By Lemma~\ref{L201} we may assume $\f\ge p\log|I_V|+O(1)$;
hence Lemma~\ref{L105} implies
$\fb_j\cdot\cO_{x,V}\supseteq\fm_{x,V}^{pj}$ for all $j\ge 1$.

From the definition of $V=V_\la$ and from Proposition~\ref{P102}
we see that 
\begin{equation*}
  \lct^\fq_y(\fb_\bullet)
  =c_y^\fq(\f)
  =\la
\end{equation*}
for every $y\in V$. 
Lemma~\ref{L101} then shows that 
\begin{equation*}
  \lct^{\fq\cdot\cO_{x,V}}(\fb_\bullet\cdot\cO_{x,V})=\la.
\end{equation*}
Recall that we assume that Conjecture~C' holds in rings of dimension
at most $n$.
Proposition~\ref{P101} implies that Conjecture~E' also holds
in rings of dimension at most $n$. We can thus find
a quasimonomial valuation
$v$ on $\cO_{x,V}$ such that
\begin{equation}\label{e119}
  \frac{A(v)+v(\fq\cdot\cO_{x,V})}{v(\fb_\bullet\cdot\cO_{x,V})}=\la.
\end{equation}

Consider a projective birational morphism 
$\pi\colon X\to\Spec\cO_{x,V}$ such that $\pi$ 
defines a log resolution of $\fq$
and such that $X$ is adapted to $v$.
Thus $v$ is given by data $Z,D,\a$ as in~\S\ref{S110}.

We analytify $\pi$ following~\S\ref{S108}  and~\S\ref{S114}. 
Let $\tau$ denote the Kiselman 
number with respect to the data $\Zan,\Dan,\a$, see~\S\ref{S107}.
We know from Proposition~\ref{P102}~(iii) and Remark~\ref{R101} that 
\begin{equation}\label{e104}
  \tau(\f\circ\pian)
  =v(\fb_\bullet\cdot\cO_{x,V}).
\end{equation}
Thus~\eqref{e119} yields
\begin{equation}\label{e120}
  \la\tau(\f\circ\pian)=A(v)+v(\fq).
\end{equation}

We use the notation from~\S\ref{S114}.
Pick a Zariski general point $z\in\Zan$.
Then 
\begin{equation*}
  \log|\fq|\circ\pian=\sum_{i=1}^m c_i\log|u_i|+O(1)
\end{equation*}
near $z$, where $c_i=\ord_{D_i}(\fq)\ge0$; see the end of~\S\ref{S108}.
We also have
\begin{equation*}
  \log|J\pian|=\sum_{i=1}^m(A_i-1)\log|u_i|+O(1),
\end{equation*}
where $A_i\in\Z_{>0}$, see~\eqref{e122}.
Finally, recall from~\eqref{e102} that
\begin{equation*}
  \f\circ\pian\le\tau(\f\circ\pian)\max_{1\le i\le m}\frac{1}{\a_i}\log|u_i|+O(1).
\end{equation*}

Fix $K\gg1$ and define disjoint open subsets $\Omega_k$,
$k\ge 0$, of $\Omega$ as in~\eqref{e118}.
As $k\to\infty$, we then have the following estimates on $\Omega_k$:
\begin{equation*}
  \f\circ\pian\le-k\tau(\f\circ\pian)+O(1),
\end{equation*}
\begin{equation*}
  \log|\fq|\circ\pian
  \ge-k\sum_{i=1}^mc_i\a_i+O(1)
  =-kv(\fq)+O(1).
\end{equation*}
Using~\eqref{e120} these estimates imply that 
\begin{equation}\label{e121}
  \la\f-\log|\fq|
  \le-k(\la\tau(\f\circ\pian)-v(\fq))+O(1)
  =-kA(v)+O(1)
\end{equation}
on the open set $\pian(\Omega_k)\subset U$.
For $1\le i\le m$ set
\begin{equation*}
  \Omega_{k,i}:=\{-(K+k)<\frac{\log|u_i|}{\a_i}<1-(K+k)\}.
\end{equation*}
Then we can estimate the volume of $\pian(\Omega_k)$ as follows:
\begin{multline*}
  \Vol\pian(\Omega_k)
  =(\sqrt{-1})^n\int_{\pian(\Omega_k)}\eta\wedge\overline\eta
  =(\sqrt{-1})^n\int_{\Omega_k}|J\pian|^2\eta_u\wedge\overline\eta_u\\
  \gtrsim\prod_{i=1}^m\sqrt{-1}\int_{\Omega_{k,i}}|u_i|^{2A_i-2}du_i\wedge d\bar{u}_i
    \gtrsim\prod_{i=1}^m\exp(-k(2A_i\a_i))
    =\exp(-2kA(v)).
  \end{multline*}
This estimate together with~\eqref{e121} concludes the proof of Theorem~D'. 
By choosing $\fq=\cO_U$ throughout all the arguments (see also 
Remark~\ref{special_equivalence}), we also obtain a proof of 
Theorem~D.
%
%
%
%
%
%
%


\end{document}